\documentclass[11pt,reqno]{amsart}

\usepackage{amsfonts, amssymb, amsthm, amsmath, bbm, wasysym, enumerate, amsxtra, latexsym, fullpage, amscd, graphics, epic, youngtab,ytableau, sansmath, mathtools, multicol, url, hyperref, bbm, blindtext, wasysym, comment, datetime, stmaryrd,soul,xcolor,import,graphicx,transparent,algorithm,algpseudocode}
\usepackage[shortlabels]{enumitem}
\usepackage[all,cmtip]{xy}
\usepackage[makeroom]{cancel}  
\usepackage{youngtab}
\usepackage{ytableau}
\usepackage{tikz}
\usetikzlibrary{cd}

\usepgflibrary{shapes.geometric}

\tikzstyle{V}=[draw, fill =black, circle, inner sep=0pt, minimum size=1.5pt]
\tikzstyle{C}=[draw, fill =white, circle, inner sep=0pt, minimum size=1.5pt]
\tikzstyle{over}=[draw=white,double=black,line width=2pt, double distance=.5pt]
\numberwithin{equation}{section}
\usepackage[margin = 0.9in]{geometry}
\theoremstyle{definition}

\newtheorem{theorem}{Theorem}[section]
\newtheorem{lemma}[theorem]{Lemma}
\newtheorem{proposition}[theorem]{Proposition}
\newtheorem{corollary}[theorem]{Corollary}
\newtheorem{definition}[theorem]{Definition}
\newtheorem{remark}[theorem]{Remark}

\newtheorem{example}[theorem]{Example}

\DeclareMathOperator{\I}{\mathcal{I}}

\def\<{\langle}
\def\>{\rangle}

\allowdisplaybreaks 

\tikzstyle directed=[postaction={decorate,decoration={markings,
    mark=at position .65 with {\arrow[arrowstyle]{stealth}}}}]
 \tikzstyle reverse directed=[postaction={decorate,decoration={markings,
    mark=at position .65 with {\arrowreversed[arrowstyle]{stealth};}}}]
    \usetikzlibrary{arrows,shapes,positioning}
\usetikzlibrary{decorations.markings}
\tikzstyle arrowstyle=[scale=1]

\title{The conflated expression graph for an arbitrary permutation}

\author[J. Zhu]{Jieru Zhu}

    \email{jieruzhu699@gmail.com}

\begin{document}

\begin{abstract}
We show that the conflated expression graph for an arbitrary permutation has a unique minimal element and a unique maximal element, and every reduced expression sits on a maximal chain from the source to the sink. This generalizes the work of Manin-Schechtman regarding higher Bruhat orders, and gives an  independent and self-contained  proof of certain results in Hothem. In addition, we give  explicit algorithms for elements in the top and bottom commutation classes. Given any reduced expression $\rho$, we  give an explicit method for producing a maximal chain containing $\rho$. 
\end{abstract}
\keywords{Higher Bruhat orders, Manin-Schechtman orientation, expression graphs}
 
\maketitle 
\section{Introduction}
The expanded expression graph, or the graph of reduced expressions, was studied as early as 1969, when Tits showed that this graph is connected \cite{Tit69}. Since then, there has been considerable interest in studying the properties of this graph \cite{Sta84,DA10,RR13}, which unravels the combinatorics of the Coxeter relations between the reduced expressions for a fixed element. Recent work of \cite{As19} shows that this graph admits a ranked poset structure,  the rank function being the co-inversion number, and has a unique maximal element with respect to the rank. 

The conflated expression graph is a quotient of the expanded expression graph once we identify all  commutator relations. It has a heated interest in categorification: the paths in the conflated expression graph induce path morphisms in Soergel calculus. There have been many interesting questions regarding the path morphisms, for example, the forking path conjecture \cite{Lib11} and the idempotent magical property coined in \cite{Eli16b}.  Understanding the path morphisms will facilitate understanding of the Karoubi envelope of the diagrammatic Hecke category, or the anti-spherical category, and ultimately of the indecomposable Soergel bimodules in this category, which descend to the Kazhdan-Lusztig basis upon decategorification.

To generalize the results of \cite{Eli16b} to the case of an arbitrary permutation, we need to establish similar properties for the conflated expression graph. The notion of higher Bruhat orders, originally studied by Manin-Schechtman \cite{MS89},   has been generalized  to an arbitrary Bruhat interval \cite{Zi93,Ho21}, and this has been further generalized to the case of the affine symmetric groups \cite{BCL26}. In this paper, we give an independent and self-contained proof for the following result: the conflated expression graph for an arbitrary permutation has a unique source and a unique sink with respect to the Manin-Schechtman order. This result highly depends on the fact that every reduced expression sits on a maximal chain from the source to the sink, a key result that we proved in Theorem~\ref{th:key}. Moreover, the number of braid moves between the source and the sink is equal to the number of $321$-patterns, i.e. the number of triples $(i,j,k)$ such that $i<j<k$ and $w_i>w_j>w_k$. This recovers certain results in \cite{Ho21,BCL26}. 

Since our approach is rather hands-on, we mention a few new results as a consequence. The proof of Theorem~\ref{th:key} gives an explicit recipe on how to produce a maximal chain containing an arbitrarily given reduced expression. Another new result is an explicit algorithm for constructing the Manin-Schechtman maximal word (see Algorithm 1), whose commutation class is maximal with respect to the Manin-Schechtman order. We also show that this class contains the super-Yamanouchi word defined in \cite{As19}. Similarly, we give an algorithm for the Manin-Schechtman minimal word and the reverse super-Yamanouchi word (see Algorithm 3), whose commutation class is minimal. These algorithms allow one to work with the source and the sink in a concrete and practical way. A non-expert shall be able to bypass the Manin-Schechtman convention and use these criteria to verify whether a word is maximal or minimal (combined with \cite[Corollary 2.10]{As19} if necessary.) A forthcoming work of the author and collaborator \cite{DZ} will define certain parabolic coset representatives that meet the reverse super-Yamanouchi criteria.

\section{The Manin-Schechtman partial order}\label{sec:MS}
We first recall some of the conventions and results in \cite{MS89}, and refer the reader to Section~2 of \cite{MS89} for more details. Let $\underline{n}=\{1,2,\dots,n\}$ and $C(n,k)$ be the set of $k$-element subsets of $\underline{n}$. Throughout this paper, we will often use the ordered tuple $(a_1,a_2,\dots,a_k)$ to denote the set $\{a_1,a_2,\dots,a_k\}$, if $1\leq a_1< a_2<\cdots < a_k \leq n$. \footnote{When the context is clear, we will also sometimes drop the commas in between.} The lexicographic ordering on $C(n,k)$ is such that $(a_1,a_2,\dots,a_k)< (b_1,b_2,\dots,b_k)$ if and only if $\exists p$, $0\leq p\leq n-1$ such that $a_i=b_i$ for $1\leq i\leq p$, and $a_{p+1}<b_{p+1}$.  A packet is the set of all $k$-element (i.e. co-cardinality one) subsets of $K$ for some $K\in C(n,k+1)$, denoted as $P(K)$. An \emph{admissible} order on $C(n,k)$ is a total order which induces either the lexicographic order or the anti-lexicographic order on any given packet. Let $A(n,k)$ be the set of all admissible orders on $C(n,k)$. Two sets $K_1,K_2\in C(n,k)$ are said to \emph{commute} if they do not belong to a common packet, i.e. $|K_1\cup K_2|\geq k+2$. We use the sequence $\rho=K_1K_2\cdots K_N$, $N=\binom{n}{k}$ to denote the total order $K_1\rho K_2\cdots \rho K_N$. For $\rho_1,\rho_2 \in A(n,k)$, we call $\rho_1,\rho_2$ equivalent if they differ by exchanging a pair of commuting neighbors of the sequence. We take the transitive closure of this relation and form equivalence classes in $A(n,k)$, denoted as $B(n,k)$. A chain with respect to $\rho$ is a subset $\mathcal{J}\subset C(n,k)$ such that if $K_1,K_2\in \mathcal{J}$ and $K_1\rho K\rho K_2$, then $K\in \mathcal{J}$.  
Suppose $K\in C(n,k+1)$, and $P(K)$ forms a chain with respect to $\rho\in A(n,k)$, and furthermore, suppose that $\rho$ induces the lexicographic order on $P(K)$. Then we define $\operatorname{inv}_K(\rho)$ to be the total order in which the elements of $P(K)$ are reversed in $\rho$, while the rest retain their positions.\footnote{In \cite{MS89}, this operator was denoted as $p_K$. Here we changed the notation to avoid confusion with the packet.} For $r\in B(n,k)$, suppose that there is a class representative $\rho$ for $r$, for which $\operatorname{inv}_K$ is defined, then define $\operatorname{inv}_K(r)$ to be the class represented by $\operatorname{inv}_K(\rho)$.  One can check that this  is independent of the chosen class representative $\rho$, i.e. $\operatorname{inv}_K$ is well-defined on $B(n,k)$. For $\rho\in A(n,k)$, Let $\operatorname{Inv(\rho)}$ be the set of all $K\in C(n,k+1)$ such that $\rho$ induces the anti-lexicographic order on $P(K)$, then one can check that $\operatorname{Inv}$ is also well-defined on $B(n,k)$. 

Let $\rho_{\operatorname{min}}$ be the lexicographic order on $C(n,k)$, and $r_{\operatorname{min}}\in   B(n,k) $ be the class it represents. Similarly, let $\rho_{\operatorname{max}}$ be the anti-lexicographic order on $C(n,k)$, and $r_{\operatorname{max}}  \in B(n,k)$ the class it represents. The main result of \cite{MS89} is as follows.
\begin{theorem}\label{th:ms}

(1) There is a well-defined partial order on $B(n,k)$, which is the transitive closure of the following relation: $r_1\leq r_2$ iff there exists $K\in C(n,k+1)$ such that $\operatorname{inv}_K$ is defined for $r_1$, and $r_2=\operatorname{inv}_K(r_1)$.

(2) $r_{\operatorname{min}} $ is the unique minimal element and  $r_{\operatorname{max}}$  is the unique maximal element in $B(n,k)$. Moreover, $B(n,k)$ forms a ranked poset, with $|\operatorname{Inv}(\cdot)|$ being the rank function. 

(3) There is a bijection between maximal chains in $B(n,k)$ and $A(n,k+1)$, given by \begin{align*}
    r_{\operatorname{min}}< \operatorname{inv}_{K_1}(r_{\operatorname{min}})< \operatorname{inv}_{K_2}\operatorname{inv}_{K_1}(r_{\operatorname{min}})<\cdots <r_{\operatorname{max}}=\operatorname{inv}_{K_N}\dots\operatorname{inv}_{K_2}\operatorname{inv}_{K_1}(r_{\operatorname{min}})\\
\mapsto K_1\cdots K_N, N=\binom{n}{k+1}.
\end{align*}

(4) Each element $r\in B(n,k)$ is uniquely determined by its inversion set $\operatorname{Inv}(r)$.

\end{theorem}

\begin{remark}\label{rmk:MS}
    It is well-known that the Manin-Schechtman conventions correspond to reduced expressions in the following way.
    \begin{itemize}
        \item $A(n,1)=B(n,1)$ correspond to the permutations in $\mathfrak{S}_n$.
        \item $A(n,2)$ corresponds to the set of all reduced expressions of $w_0$, the longest word in $\mathfrak{S}_n$. These are the vertices in the expanded expression graph. 
        \item $B(n,2)$ corresponds to the set of commutation classes of the reduced expressions of $w_0$. These are the vertices in the conflated expression graph. 
        \item For $|K|=2$ and $\rho \in A(n,1)$, $\operatorname{inv}_K(\rho)$ correspond to commutator relations.
        \item For $|K|=3$ and $\rho \in A(n,2)$, $\operatorname{inv}_K(\rho)$  correspond to braid moves. 
    \end{itemize}
\end{remark}    

Define the Manin-Schechtman graph $\Gamma$ as the graph whose vertices are given by $A(n,1)$, and the directed edges given by the operation $\operatorname{Inv}_K$. Namely, there is a directed edge $\rho_1 \to \rho_2$ labeled $K$, iff there exits $K\in C(n,2)$ such that $\rho_2=\operatorname{inv}_{K}(\rho_1)$. In particular, this records the ranked poset structure on $A(n,1)=B(n,1)$.

\begin{example}\label{ex:1}
We give the Manin-Schechtman graph for $\mathfrak{S}_4$ as follows. Here, vertices are permutations. The edges indicate the operation $\operatorname{inv}_K(\rho)$ for $|K|=2$, which is equivalent to a simple transposition. A directed path from $\rho_{\operatorname{min}}=1234$ to $\rho_{\operatorname{max}}=4321$ indicates a maximal chain in $A(n,1)$, which is equivalent to the sequence of its edges, i.e. a reduced expression. 

Braid moves work as follows: for example, if $\rho=(12)\textcolor{red}{(13)(14)(34)}(24)(23)$ and $K=(134)$, then $\operatorname{inv}_K(\rho)=(12)\textcolor{red}{(34)(14)(13)}(24)(23)$. This move is indicated by the red hexagon in the graph.
\end{example}
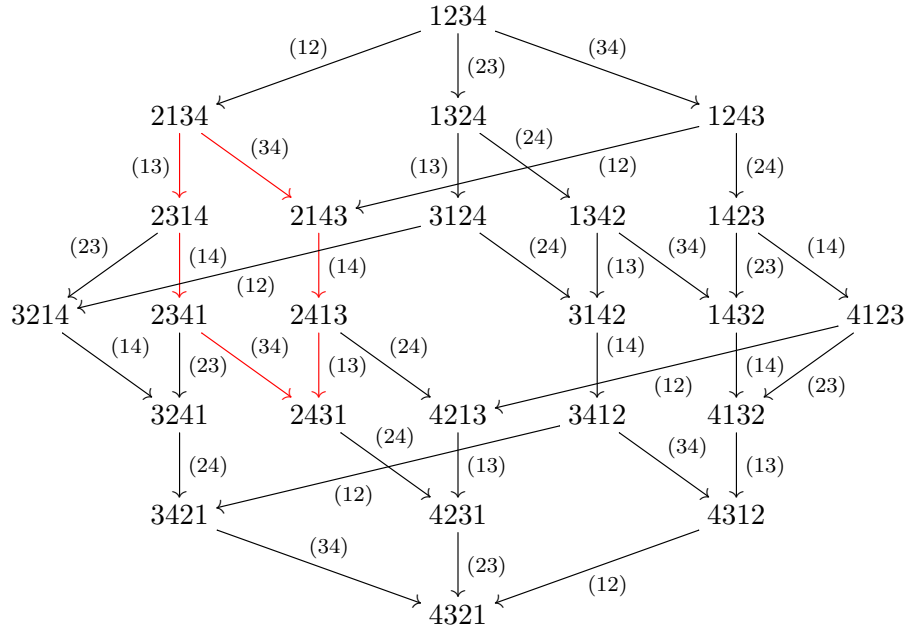
\begin{center}
\leavevmode
    \xymatrix{
&&&1234 \ar[dll]_{(12)}\ar[d]^{(23)}\ar[drr]^{(34)} &&&\\
&2134 \ar@[red][d]_{(13)}\ar@[red][dr]^{(34)}&&1324 \ar[d]_{(13)}\ar[dr]^(.4){(24)}&&1243 \ar[dlll]^(.3){(12)}\ar[d]^{(24)}&\\
&2314 \ar[dl]_{(23)}\ar@[red][d]^(.4){(14)}&2143 \ar@[red][d]^{(14)} &3124 \ar[dlll]^{(12)} \ar[dr]^{(24)}&1342 \ar[d]^{(13)} \ar[dr]^{(34)}&1423 \ar[d]^{(23)} \ar[dr]^{(14)}&\\
3214 \ar[dr]^{(14)}&2341 \ar[d]^{(23)}\ar@[red][dr]^{(34)}&2413 \ar@[red][d]^{(13)}\ar[dr]^{(24)} &&3142 \ar[d]^(.3){(14)}&1432 \ar[d]^{(14)}&4123 \ar[dlll]^{(12)}\ar[dl]^{(23)}\\
&3241 \ar[d]^{(24)}&2431 \ar[dr]^(.4){(24)} &4213 \ar[d]^{(13)}&3412 \ar[dlll]^(.6){(12)} \ar[dr]^{(34)} &4132 \ar[d]^{(13)}&\\
&3421 \ar[drr]^{(34)}&&4231 \ar[d]^{(23)} &&4312\ar[dll]^{(12)} &\\
&&&4321&&&}
\end{center}

\section{The general case}
Our next goal is to extend the above notions to an element $w$ that is not the longest word in $\mathfrak{S}_n$. Define the set of inversions $\operatorname{Inv}(w)=\{(w_j,w_i)\mid 1\leq i<j \leq n, w_i>w_j\}$. Then the following result is straightforward.

\begin{lemma}\label{lem:subgraph}
Let $\Gamma_{w}$ be the  subgraph of $\Gamma$, whose vertices are all elements $y\in \mathfrak{S}_n$ such that $y\leq w$ in Bruhat order. Moreover, there is an edge $a\to b$ in $\Gamma_w$ iff there is an edge $a\to b$ in $\Gamma$.  Then $\Gamma_{w}$ is also a ranked poset with a unique minimal element $\operatorname{id}=123\cdots n \in \mathfrak{S}_n$, a unique maximal element $w\in \mathfrak{S}_n$, and any path from $\operatorname{id}$ to $w$ has the set of edges equal to $\operatorname{Inv}(w)$.
\end{lemma}

Our goal is to study the higher morphisms in $\Gamma_w$. Let $A(n,1;w)=\operatorname{Inv}(w)$. 
Let $A(n,2;w)$ be the set of the total orders on $\operatorname{Inv}(w)$ that represent the maximal paths in $\Gamma_w$. In precise terms, this is the set of elements $\rho=K_1\cdots K_N$, $K_i\in \operatorname{Inv}(w)$, $N=|\operatorname{Inv}(w)|$, so that $\operatorname{inv}_{K_N}\cdots \operatorname{inv}_{K_1}(12\cdots n)$ is well-defined. Define a $321$-\emph{pattern} as a set $\{w_i,w_j,w_k\}$ such that $i<j<k$ and $w_i>w_j>w_k$, and we sometimes write it as an ordered tuple $(w_k,w_j,w_i)$. Similarly, we define $132$-patterns, $213$-patterns, etc,  accordingly. Let $\mathcal{I}(w)$ be the set of 321-patterns in $w$. The following result is straightforward.
\begin{lemma}\label{lem:321}
   Let $K\in C(n,3)$, then $K\in \mathcal{I}(w)$ iff $P(K)\subset \operatorname{Inv}(w)$.
\end{lemma}

Theorem~\ref{th:ms}(3), once specialized to the case of $k=1$, implies that any reduced expression for $w_0$ (i.e. a maximal chain in the Manin-Schechtman graph $\Gamma$) corresponds to an admissible order on $C(n,2)$. We will prove the analog of this for $\Gamma_w$.

\begin{lemma}\label{lem:P3}
    Let $\rho \in A(n,2;w)$, then $\rho$ induces either the lexicographic order or the anti-lexicographic order on $P(K)\cap \operatorname{Inv}(w)$, for any given $K\in C(n,3)$. 
\end{lemma}
\begin{proof}
    With Lemma~\ref{lem:321} in mind, the claim is only nontrivial when $K \in \mathcal{I}(w)$, otherwise $|P(K)\cap \operatorname{Inv}(w)|\leq 2$.  In particular, let $(a,b,c)$ be a $321$-pattern in $w$. Then the following steps to produce $w$ from $\operatorname{id}$:
\begin{align}
  \cdots  a\cdots b\cdots c \cdots \rightarrow    \cdots  abc \cdots \rightarrow \cdots  a cb  \cdots  \rightarrow \cdots  cab  \cdots  \rightarrow \cdots  cba  \cdots   \rightarrow \cdots  c\cdots b \cdots a  \cdots \label{eq:b1}
 \end{align}
 are given by applying $\operatorname{inv}_K$ to the following sequence of sets $K$:
 \begin{align}
     \cdots (bc)(ac)(ab)\cdots \label{A1}
 \end{align}
where $\cdots$ indicates the sequence of moves to make $a,b,c$ adjacent to each other and then apart to their designated positions. Alternatively, one could have the following sequence of moves:
\begin{align}
  \cdots  a\cdots b\cdots c \cdots  \rightarrow  \cdots  abc \cdots \rightarrow \cdots  bac  \cdots  \rightarrow \cdots  bca  \cdots  \rightarrow \cdots  cba  \cdots   \rightarrow \cdots  c\cdots b \cdots a  \cdots  \label{eq:b2}
 \end{align}
 which is given by applying $\operatorname{inv}_K$ for the following sequence of sets $K$:
 \begin{align}
     \cdots (ab)(ac)(bc)\cdots \label{A2}
 \end{align}
 The packet $P(K)$ for $K=(abc)$ is ordered anti-lexicographically in (\ref{A1}) and lexicographically in (\ref{A2}), and hence the claim follows.
 
\end{proof}

\begin{remark}\label{rem:trick}
There is indeed a quicker proof: one can simply extend the reduced expression of $w$ to that of $w_0$, and because the latter corresponds to an admissible order on $C(n,2)$, the restriction on $\operatorname{Inv}(w)$ satisfies the assumptions in the claim. The proof here is useful for Lemma~\ref{lem:same}, and also for the purpose of showing the following trick: if one examines the Manin-Schechtman graph for $\mathfrak{S}_3$ (c.f. \cite[Section 2]{MS89}), take either of the maximal chains and replace $1,2,3$ with $a,b,c$, the order among $P(K')$ for $K'=(123)$ is the same as the order of $P(K)$ in the two cases above. This means that it suffices to check that all maximal chains for $\mathfrak{S}_3$ are admissible. The relative positions of $a,b,c$ will follow the behavior of $1,2,3$, and so will the corresponding packets. This is because if $L \notin P(K), L\in C(n,2)$, then $\operatorname{inv}_L$  does not change the relative positions of $a,b,c$, and therefore can be omitted from the discussion. We will be using this trick again later when we examine $\mathfrak{S}_4$. 

\end{remark}

Contrary to Theorem~\ref{th:ms}(3), the reverse is not true, i.e the criterion in Lemma~\ref{lem:P3} is not enough to guarantee that $\rho$ corresponds to a reduced expression.  The following lemma gives some additional criteria. The reader may refer to the notion of realizable sets in \cite{Zi93}.

\begin{lemma}\label{lem:describe}
   Let $\rho\in A(n,2;w)$,  then the following is also satisfied:
    

    (1) If $K$ is a $312$-pattern in $w$, then $P(K)\cap \operatorname{Inv}(w)$ is ordered anti-lexicographically in $\rho$.

    (2) If $K$ is a $231$-pattern in $w$, then $P(K)\cap \operatorname{Inv}(w)$ is ordered lexicographically in $\rho$.
\end{lemma}
\begin{proof}

    We first examine a $312$-pattern. In this case, suppose $\{c,a,b\}$ is a $312$-pattern with $a<b<c$. Notice that $\rho$ can be extended to a maximal path $\rho'$ in $\Gamma$. Then $(ac),(bc)\in \operatorname{Inv}(w)$, and $(ab)$ appears after $\rho$ in $\rho'$. Since $(ab)$ is the smallest element in $P((abc))$ and appears last in $\rho'$, the order induced on $P((abc))$ must be anti-lexicographic, hence $(bc),(ac)$ must appear in this order in $\rho$. The other case can be shown similarly.
\end{proof}

In \cite{As19}, the author defined a partial order on the exapanded expression graph, whose vertices are all reduced expressions of a fixed permutation $w$, and the edges indicate the Coxeter relations. The conflated expression graph is the result of contracting all reduced expressions in the same commutation class to a single vertex. The proof of \cite[Theorem~2.8]{As19} implies the following formulation of the partial order:
\begin{definition}\label{def:partial}
Let $\leq$ be the transitive closure of the following relation: let $r_1$, $r_2$ be commutation classes of reduced expressions for $w$, $r_1 \leq r_2$  iff there exist $\rho_1 \in r_1$, $\rho_2 \in r_2$, and $\rho_2$ can be obtained from $\rho_1$ via  a braid move $s_i s_{i+1} s_i \to s_{i+1}s_i s_{i+1}$  for some $i$.
\end{definition}

Recall the equivalence relation defined on $A(n,k)$ in Section~\ref{sec:MS}: $\rho_1 \sim \rho_2$ if and only if $\rho_1$ and $\rho_2$ can be obtained from a sequence of exchanging commuting neighbors. Remark~\ref{rmk:MS} implies that equivalence classes in $A(n,2)$ are closed in $A(n,2;w)$, therefore we can let $B(n,2;w)$ be the set of equivalence classes in $A(n,2;w)$. The Manin-Schechtman partial order (cf. Theorem~\ref{th:ms}), defined on $B(n,2)$, also passes down to $B(n,2;w)$ via restriction, with the same rank function $|\operatorname{Inv}(\cdot)|$. Note here, $$\operatorname{Inv}(\rho)=\{K\in C(n,3)\mid P(K)\text{ is ordered anti-lexicographically in }\rho\}.$$
The following result is folklore (for example, see \cite{Eli16b} for the longest word). Here we give a proof for the general case.

\begin{lemma}\label{lem:same}

The Manin-Schechtman partial order is equivalent to Definition~\ref{def:partial}.

\end{lemma}
\begin{proof}
    One simply observes that the expression $\cdots s_{i+1}s_is_{i+1}\cdots$ appears in (\ref{eq:b1}), and the packet $P(K)$ for $K=(abc)$ is ordered anti-lexicographically in (\ref{A1}). On the other hand,   the expression $\cdots s_i s_{i+1} s_i \cdots $ appears in (\ref{eq:b2}), and the packet $P(K)$ for $K=(abc)$ is ordered lexicographically in (\ref{A2}). In the Manin-Schechtman order, (\ref{A1}) is the result of applying $\operatorname{inv}_K$ to $(\ref{A2})$, therefore the order goes up, which is consistent with the Assaf order, if we compare (\ref{eq:b1}) against  (\ref{eq:b2}).
\end{proof}



Fix a permutation $w\in \mathfrak{S}_n$. Let $\rho_{\operatorname{max}}\in A(n,2;w)$ be the order defined by the following algorithm, and we will call it the Manin-Schechtman maximal word, or MSMax word for short.

\begin{algorithm}[ht]
    \begin{algorithmic}[1]
    \Procedure{MSMax}{$w$}
    \State $v \gets 123\cdots n$
    \State $\rho \gets ()$
    \State $N=n$
    \While {$v \neq w$}
    \State $k \gets \max\{j \mid  (v_j,v_{j+1})\in \operatorname{Inv}(w), j< N \}$
    \State $N \gets k$
\While{$(v_k,v_{k+1})\in \operatorname{Inv}(w) \land k<n$ }
\State $\rho \gets \rho,(v_k,v_{k+1})$
\State $v \gets \operatorname{inv}_{(v_k,v_{k+1})}(v)$
\State $k \gets k+1$
\EndWhile
    \EndWhile
    \State \textbf{return} $\rho$
    \EndProcedure
    \end{algorithmic}
    \caption{Manin-Schechtman maximal word}
\end{algorithm}

\begin{example}\label{ex:main}
We use the permutation $w=41758236$ to compare with \cite[Example D]{As19}. We highlight the newly added sets in each loop, and also the position $N$ at the end of each loop.

[loop 1:] $N=6$, $\rho=(67)(68)$, $v=12345\underline{7}86$;

[loop 2:] $N=5$, $\rho=(67)(68)\boldsymbol{(57)}$, $v=1234\underline{7}586$; 

[loop 3:] $N=3$,  $\rho=(67)(68)(57)\boldsymbol{(34)(37)(35)(38)}$, $v=12\underline{4}75836$;

[loop 4:] $N=2$, $\rho=(67)(68)(57)(34)(37)(35)(38)\boldsymbol{(24)(27)(25)(28)}$, $v=1\underline{4}758236$;

[loop 5:] $N=1$, $\rho=(67)(68)(57)(34)(37)(35)(38)(24)(27)(25)(28)\boldsymbol{(14)}$, $v=\underline{4}1758236$.

Notice that this corresponds to the reduced expression $s_1s_5s_4s_3s_2s_6s_5s_4s_3s_5s_7s_6$. One may use \cite[Corollary 2.10]{As19} to check that this is in the same commutation class as the super-Yamanouchi word, a result we will prove in general.
\end{example}

\begin{lemma}\label{lem:good}
The algorithm is well-defined and eventually stops.
\end{lemma}
\begin{proof}
This follows from the fact: at the end of each loop, the entries $v_{N}, v_{N+1}, \dots, v_n$ retain their relative positions in $w$. In other words, $v_i$ appears before $v_{i+1}$  in $w$, for all $N\leq i \leq n-1$. Moreover, $v_i=i$ for $i<N$, and $w_j\geq N$ for $j\geq N$.

We prove this by induction, and the base case is when $N=n$ initially, which is automatically true. Now we prove the induction step: i.e. assume that the above claim holds at the end of the previous loop, we would like to show that it holds at the end of the next loop also. Suppose $N_0$ is the $N$-value at the end of the previous loop and $N_1$ is the newly assigned $N$-value. Then at the beginning of this loop, the sequence $(v_{N_1+1},\cdots, v_{N_0-1})=(N_1+1,\cdots, N_0-1)$ must also retain their relative positions as in $w$, otherwise there would be an inversion among these entries, creating an inversion after location $N_ 1$ and violating the maximal property.

Since $v_{N_0-1}$ and $v_{N_0}$ must also retain their relative position in $w$ for the same reason, we have shown that $(v_{N_1+1},\cdots,v_n)$ must all retain their relative positions in $w$ based on the inductive hypothesis.

The loop simply performs the following task: suppose $v_{N_1}$ appears after $v_j$ and before $v_{j+1}$ in $w$, for some $N_1<j\leq  n$. Since $v_{N_1}=N-1$ will form an inversion with exactly $v_{N_1+1},\dots,v_j$, the loop will move $v_{N_1}$ past exactly those elements in $v$ and stop between $v_j$ and $v_{j+1}$. Hence the relative positions among $v_j$, $v_{N_1}$ and $v_{j+1}$ also get preserved, and hence the claim holds at the end of this loop also.
\end{proof}

\begin{lemma}
The Manin-Schechtman maximal word $\rho_{\operatorname{max}}$ admits the inversion number $|\operatorname{Inv}(\rho_{\operatorname{max}})|=|\I(w)|$.
\end{lemma}
\begin{proof}
    Observe that the sets appearing in $\rho_{\operatorname{max}}$ can be divided into segments $\rho^{(1)},\rho^{(2)},\dots,\rho^{(t)}$, where $\rho^{(i)}$ is the sequence of sets $K$ added in loop $i$. In particular, each set in $\rho^{(i)}$ will be in the form of $(N,M)$ for the fixed $N$-value in this loop. This $N$ is precisely the $v_N$ that gets repeatedly moved in the proof of Lemma~\ref{lem:good}. Since the $N$-value would decrease after each loop, the first letter of each set $K$ which appears in $\rho$ is weakly decreasing.

     Hence if there exists $K\in\mathcal{I}$, $K=(abc)$, then $P(K)$ appears either in the sequence of $(ab)(ac)(bc)$ or  $(bc)(ac)(ab)$. But the former will have the first letter appearing in an increasing order, hence creating a contradiction. Therefore, $P(K)$ must appear in the order of  $(bc)(ac)(ab)$, which means that $K\in \operatorname{Inv}(\rho_{\operatorname{max}})$ for all $K\in \I(w)$. 
\end{proof}

In \cite{As19}, the author gave an algorithm for producing the unique super-Yamanouchi word, which is the unique maximal element with respect to the co-inversion number. In the definition below, we use the conventional way of writing a reduced expression in terms of simple transpositions, and only write the subscript $i$ for $s_i$. In other words, $s_1s_2s_1s_3$ becomes $1213$. \footnote{We note that the Manin-Schechtman convention reads from left to right and simple transitions read from right to left, in terms of the order of operation. This should be kept in mind when comparing conventions.}
\begin{definition}
    A reduced expression $\rho$ is super-Yamanouchi if and only if it has a decomposition $\rho=(\rho^{(1)}\mid\rho^{(2)}\mid\cdots\mid\rho^{(t)})$ such that each $\rho^{(i)}$ is an increasing interval, and the beginning of $\rho^{(i)}$ is larger than the beginning of $\rho^{(i+1)}$, $1\leq i\leq t-1$.
\end{definition}
The algorithm in \cite{As19} becomes the following, once we use the Manin-Schechtman convention for a reduced expression.

\begin{algorithm}[ht]
    \begin{algorithmic}[1]
    \Procedure{Super}{$w$}
    \State $v \gets w$
    \State $\pi \gets ()$
    \While {$\ell (v)>0$} 
   \State $j \gets \operatorname{max}\{k \mid v_k>v_{k+1}\}$
   \While {$v_j > v_{j+1} \land j<n$}
   \State $\pi \gets (v_{j+1},v_j), \pi $
   \State $v \gets \operatorname{inv}_{(v_{j+1},v_j)}(v)$
   \State $j \gets j+1 $ 
   \EndWhile
    \EndWhile
    \State \textbf{return} $\pi$
    \EndProcedure
    \end{algorithmic}
    \caption{Super-Yamanouchi word}
\end{algorithm}
\begin{example}\label{ex:2}
Continuing from Example~\ref{ex:main}, we use the permutation $w=41758236$. We also highlight the letter that has been moved during each loop.

[loop 1:] $j=5$, $\pi=(68)(38)(28)$, $v=4175236\underline{8}$;

[loop 2:] $j=4$, $\pi=\boldsymbol{(35)(25)}(68)(38)(28)$, $v=41723\underline{5}68$; 

[loop 3:] $j=3$,  $\pi=\boldsymbol{(67)(57)(37)(27)}(35)(25)(68)(38)(28)$, $v=412356\underline{7}8$;

[loop 4:] $j=1$, $\pi=\boldsymbol{(34)(24)(14)}(67)(57)(37)(27)(35)(25)(68)(38)(28)$, $v=123\underline{4}5678$.

One may check manually that this is the same equivalent class as the Manin-Schechtman maximal word computed in Example~\ref{ex:main}.
\end{example}

\begin{lemma}
    The super-Yamanouchi word $\pi_{\operatorname{max}}$ also has the inversion number $\operatorname{Inv}(\pi_{\operatorname{max}})=|\I(w)|$.
\end{lemma}
\begin{proof}
    Suppose $K=\{c,b,a\}$ is a $321$-pattern in $w$, with $a<b<c$. By the algorithm, the inversion $(a,b)$ is dealt with first, and therefore sits at the rightmost position among $P(K)$ in $\pi_{\operatorname{max}}$. Therefore any $K\in \I(w)$ is ordered anti-lexicographically in $\pi_{\operatorname{max}}$, the claim then follows.
\end{proof}

\begin{proposition}
    Suppose $r\in B(n,2;w)$ and $|\operatorname{Inv}(r)|=|\I(w)|$, then $r$ is the unique maximal element with respect to the rank function $|\operatorname{Inv}(\cdot)|$.
\end{proposition}
\begin{proof}
    The maximal property comes from Lemma~\ref{lem:321}, here we show the uniqueness. Let $\rho_1,\rho_2\in A(n,2;w)$ with the property that $|\operatorname{Inv}(\rho_1)|=|\operatorname{Inv}(\rho_2)|=|\I(w)|$, then we claim that they are in the same commutation class. We only need to compare the relative positions of sets which don't commute, i.e. sets which belong to a common packet. Suppose a $K_1$ and $K_2$ are two such sets, $|K_1\cup K_2|=3$, $K=K_1\cup K_2$, then we discussed by cases.

    Suppose $K$ is a $321$-pattern in $w$, then by assumption, $K\in \I(w) =\operatorname{Inv}(\rho_1)=\operatorname{Inv}(\rho_2)$. Therefore, all elements of $P(K)$  are ordered anti-lexicographically in both $\rho_1$ and $\rho_2$.

    A $123$-pattern $L$ will have its packet $P(L)$ disjoint with $\operatorname{Inv}(w)$, hence $K$ cannot be a $123$-pattern.  Similarly, a $132$- or $213$-pattern $L$ will have its packet $P(L)$ intersecting $\operatorname{Inv}(w)$ with only one element, hence $K$ cannot be either. The only cases left are discussed in Lemma~\ref{lem:describe}, for which the relative positions between $K_1$ and $K_2$ are already determined. Our claim follows. 
\end{proof}
\begin{corollary}
    Both $\rho_{\operatorname{max}}$ and $\pi_{\operatorname{max}}$ lie in the same commutation class $r_{\operatorname{max}}$, and $r_{\operatorname{max}}$ is maximal with respect to the rank function $|\operatorname{Inv}(\cdot)|$.
\end{corollary}

\begin{remark}
    To clarify the abusive use of the word ``maximal'': the above result  only implies that if $r_{\operatorname{max}} \leq r$, then $r_{\operatorname{max}} = r$. However, it is not true a priori that $r_{\operatorname{max}} \geq r$ for all $r\in B(n,2;w)$, because it is not yet clear that $r_{\operatorname{max}}$ is comparable to any $r\in B(n,2;w)$. In other words, we do not yet know that in the conflated expression graph, any path from $r_{\operatorname{max}}$ to $r$ has a consistent orientation. We will show this later.
\end{remark}

Let $t\in \mathfrak{S}_n, t(i)=n+1-i$ for $1\leq i\leq n$ be the map that reverses the letter set, let $\tau: C(n,k) \to C(n,k)$ be the map that it induces. Also, let $\tau: A(n,2;w) \to A(n,2;w)$ be its extension, i.e., $\tau(K_1 K_2\cdots K_N)=\tau(K_1)\tau(K_2)\cdots\tau(K_N)$. Let $\rho_{\operatorname{max}}(w)$ be the unique Manin-Schechtman maximal word associated with $w\in  \mathfrak{S}_n$ and define $\rho_{\operatorname{min}}=\tau(\rho_{\operatorname{max}}(twt))$ to be the \emph{Manin-Schechtman minimal} word associated with $w$. 

Similarly, let $\pi_{\operatorname{max}}(w)$ be the unique super-Yamanouchi word associated with $w\in  \mathfrak{S}_n $ and define $\pi_{\operatorname{min}}=\tau(\pi_{\operatorname{max}}(twt))$ to be the \emph{reverse super-Yamanouchi} word associated with $w $. Then one can dualize Algorithm 1, Algorithm 2, and all the subsequent proofs, by reversing the letter set and the indexing set if necessary, and one obtains the following results.

\begin{proposition}
(1) Both $\rho_{\operatorname{min}}$ and $\pi_{\operatorname{min}}$  lie in the same commutation class $r_{\operatorname{min}}\in B(n,2;w)$, and $r_{\operatorname{min}}$ is the unique minimal word with respect to the rank function with $|\operatorname{Inv}(r)|=0$.

(2) Written in terms of simple transpositions, $\pi_{\operatorname{min}}$ is the unique word subject to the following condition: it has a decomposition $\pi=(\pi^{(1)}\mid\pi^{(2)}\mid\cdots\mid\pi^{(t)})$ such that each $\pi^{(i)}$ is a decreasing interval, and the beginning of $\rho^{(i)}$ is smaller than the beginning of $\rho^{(i+1)}$, for $1\leq i\leq t-1$.
\end{proposition}

In particular, the dual version of \cite[Algorithm 1]{As19} becomes the following. This allows one to construct the reverse super-Yamanouchi word in terms of simple transpositions.
\begin{algorithm}[ht]
    \begin{algorithmic}[1]
    \Procedure{reveres-Super}{$w$}
    \State $v \gets w$
    \State $\pi \gets ()$
    \While {$\ell (v)>0$} 
   \State $j \gets \operatorname{min}\{k \mid v_k<v_{k-1}\}$
   \While {$v_j < v_{j-1} \land j>1$} 
   \State $\pi \gets  (\pi, s_{j-1}) $
   \State $v \gets s_{j-1}(v)$
   \State $j \gets j-1 $ 
   \EndWhile
    \EndWhile
    \State \textbf{return} $\pi$
    \EndProcedure
    \end{algorithmic}
    \caption{Reverse Super-Yamanouchi word}
\end{algorithm}

\begin{example}\label{ex:3}
    Continuing from Example~\ref{ex:2}, we use the permutation $w=41758236$. We have $twt=36714285$ and $\rho_{\operatorname{max}}(twt)=(56)(57)(58)(46)(47)(23)(26)(27)(24)(13)(16)(17)$. Therefore, $\rho_{\operatorname{min}}=\tau(\rho_{\operatorname{max}}(twt))=(34)(24)(14)(35)(25)(67)(37)(27)(57)(68)(38)(28)$. Notice that the second letter appears increasingly in each set.

    On the other hand, $\pi_{\operatorname{max}}(twt)=(23)(13)(56)(46)(26)(16)(37)(47)(27)(17)(24)(58)$, and $\pi_{\operatorname{min}}=\tau(\pi_{\operatorname{max}}(twt))=(67)(68)(34)(35)(37)(38)(24)(25)(27)(28)(57)(14)$. Written in terms of simple transpositions, $\pi_{\operatorname{min}}=s_{1}\mid s_{3} \mid s_{5}s_{4}s_{3}s_{2}\mid s_{6}s_{5}s_{4}s_{3} \mid s_{7}s_{6}$.
\end{example}

\begin{theorem}\label{th:key}
Suppose $\rho \in A(n,2;w)$, $|\operatorname{Inv}(\rho)|<|\I(w)|$. Then there exists $\rho'$ equivalent to $\rho$, $K \in \I(w)\backslash \operatorname{Inv}(\rho)$, such that $P(K)$ forms a chain in $\rho'$. In other words, $\operatorname{inv}_K(\rho')$ is well-defined.
\end{theorem}
\begin{proof}
    We choose $L\in \I(w)\backslash \operatorname{Inv}(\rho)$ which admits the \emph{closest together} property in the following sense: suppose $\rho=K_1K_2\cdots K_N$, $P(L)=\{K_{i_1},K_{i_2},K_{i_3}\}$ with $1\leq i_1<i_2<i_3\leq N$, then there does not exist $L'\in \I(w)\backslash \operatorname{Inv}(\rho)$, $L\neq L'$, $P(L')=\{K_{j_1},K_{j_2},K_{j_3}\}$ with $i_1\leq j_1<j_2<j_3\leq i_3$. i.e. there does not exist another packet lying nested inside $P(L)$. Notice that if such $L'$ exists, then the qualities in $i_1\leq j_1$ and $j_3\leq i_3$ cannot be both achieved, otherwise $L'=L=K_{i_1}\cup K_{i_3}$.

    Recall that $M_1,M_2 \in C(n,2)$ are said to commute if they don't belong to a common packet. We now claim that $P(L)$ can be put into a chain by moving commuting neighbors past each other. Suppose on the contrary, that this cannot be done. Then one of the following must be true (1) There exists $K'$ between  $K_{i_1},K_{i_2}$ which cannot be commuted past $K_{i_1}$ to the left, nor can it be commuted past $K_{i_2}$ to the right, or (2) There exists $K'$ between  $K_{i_2},K_{i_3}$ which cannot be commuted past $K_{i_2}$ to the left, nor can it be commuted past $K_{i_3}$ to the right, or (3)  There exists $K'$ between  $K_{i_1},K_{i_3}$ which cannot be commuted past $K_{i_1}$ to the left, nor can it be commuted past $K_{i_3}$ to the right, but $K'$ commutes with $K_{i_2}$.

    Let $L=(abc)$ with $a<b<c$. Then $K_{i_1}=(ab)$, $K_{i_2}=(ac)$, $K_{i_3}=(bc)$. Suppose we are in case (1) above. Then $K'$ has to be in the form of $\{a,d\}$ for some $d\notin\{a,b,c\}$. We further discuss by cases: (i) $d<a$, (ii) $a<d<b$, (iii) $b<d<c$, (iv) $c<d$.

 Suppose we are in case (i), then the four sets appear in the order $(ab)(da)(ac)(bc)$. By using the trick in Remark~\ref{rem:trick}, they are modeled by the behavior of $(23)(12)(24)(34)$ in $\mathfrak{S}_4$. By manually checking all the paths in Example~\ref{ex:1}, the only  path where these four sets appear in this order, is the path $(23)(13)(12)(14)(24)(34)$. In this case, $P((124))$ is another packet nested inside $P((234))$, i.e. $P((dbc))=\{(db)(dc)(bc)\}$ is another packet nested inside $P(L)$, violating the closest together assumption for $L$. We reach a contradiction.


Suppose we are in case (ii), then the four sets appear in the order $(ab)(ad)(ac)(bc)$. By using the trick in Remark~\ref{rem:trick}, they are modeled by the behavior of $(13)(12)(14)(34)$ in $\mathfrak{S}_4$. By manually checking all the paths in Example~\ref{ex:1}, the only  path where these four sets appear in this order, is the path $(23)(13)(12)(14)(24)(34)$. In which case, $P((124))$ is another packet nested inside $P((134))$, i.e. $P((adc))=\{(ad)(ac)(dc)\}$ is another packet nested inside $P(L)$, violating the closest together assumption for $L$. We reach a contradiction.

    Suppose we are in case (iii), then the four sets appear in the order $(ab)(ad)(ac)(bc)$. By using the trick in Remark~\ref{rem:trick}, they are modeled by the behavior of $(12)(13)(14)(24)$ in $\mathfrak{S}_4$. By manually checking all the paths in Example~\ref{ex:1}, the only  paths where these four sets appear in this order, are the paths $(12)(13)(14)(23)(24)(34)$ or $(12)(13)(23)(14)(24)(34)$. In either case, $P((134))$ is another packet nested inside $P((124))$, i.e. $P((adc))=\{(ad)(ac)(dc)\}$ is another packet nested inside $P(L)$, violating the closest together assumption for $L$. We reach a contradiction.

    Suppose we are in case (iv), then the four sets appear in the order $(ab)(ad)(ac)(bc)$. By using the trick in Remark~\ref{rem:trick}, they are modeled by the behavior of $(12)(14)(13)(23)$ in $\mathfrak{S}_4$. By manually checking all the paths in Example~\ref{ex:1}, the only  paths where these four sets appear in this order, are the paths $(12)(34)(14)(13)(24)(23)$ or $(12)(34)(14)(24)(13)(23)$. In either case, $P((124))$ is another packet nested inside $P((123))$, i.e. $P((abd))=\{(ab)(ad)(bd)\}$ is another packet nested inside $P(L)$, violating the closest together assumption for $L$. We reach a contradiction.

    Case (2) is symmetric to Case (1) and can be checked by a similar argument. We now discuss Case (3). $K'$ must be equal to $\{b,d\}$ for some $d\notin\{a,b,c\}$ and we can assume it lies between $K_{i_1}$ and $K_{i_2}$. We still discuss by cases: (i) $d<a$, (ii) $a<d<b$, (iii) $b<d<c$, (iv) $c<d$. The arguments are similar as above and we include them for completeness nevertheless.

    Suppose we are in case (i), then the four sets appear in the order $(ab)(db)(ac)(bc)$. By using the trick in Remark~\ref{rem:trick}, they are modeled by the behavior of $(23)(13)(24)(34)$ in $\mathfrak{S}_4$. By manually checking all the paths in Example~\ref{ex:1}, the only  paths where these four sets appear in this order, are the paths $(23)(13)(24)(14)(34)(12)$ or $(23)(13)(24)(14)(12)(34)$. In either case, $P((134))$ is another packet nested inside $P((234))$, i.e. $P((dbc))=\{(db)(dc)(bc)\}$ is another packet nested inside $P(L)$, violating the closest together assumption for $L$. We reach a contradiction.

    Suppose we are in case (ii), then the four sets appear in the order $(ab)(db)(ac)(bc)$. By using the trick in Remark~\ref{rem:trick}, they are modeled by the behavior of $(13)(23)(14)(34)$ in $\mathfrak{S}_4$. By manually checking all the paths in Example~\ref{ex:1}, the only  path where these four sets appear in this order, is the path $(12)(13)(23)(14)(24)(34)$. In this case, $P((123))$ is another packet nested inside $P((134))$, i.e. $P((adb))=\{(ad)(ab)(db)\}$ is another packet nested inside $P(L)$, violating the closest together assumption for $L$. We reach a contradiction.

    Suppose we are in case (iii), then the four sets appear in the order $(ab)(bd)(ac)(bc)$. By using the trick in Remark~\ref{rem:trick}, they are modeled by the behavior of $(12)(23)(14)(24)$ in $\mathfrak{S}_4$. By manually checking all the paths in Example~\ref{ex:1}, the only  path where these four sets appear in this order, is the path $(12)(13)(23)(14)(24)(34)$. In this case, $P((123))$ is another packet nested inside $P((124))$, i.e. $P((abd))=\{(ab)(ad)(bd)\}$ is another packet nested inside $P(L)$, violating the closest together assumption for $L$. We reach a contradiction.

     Suppose we are in case (iv), then the four sets appear in the order $(ab)(bd)(ac)(bc)$. By using the trick in Remark~\ref{rem:trick}, they are modeled by the behavior of $(12)(24)(13)(23)$ in $\mathfrak{S}_4$. By manually checking all the paths in Example~\ref{ex:1}, the only  paths where these four sets appear in this order, are the paths $(34)(12)(14)(24)(13)(23)$ and $(34)(12)(14)(13)(24)(23)$. In either case, $P((124))$ is another packet nested inside $P((123))$, i.e. $P((abd))=\{(ab)(ad)(bd)\}$ is another packet nested inside $P(L)$, violating the closest together assumption for $L$. We reach a contradiction.
\end{proof}

\begin{corollary}
    The class $r_{\operatorname{max}}$ is the unique maximal element in $B(n,2;w)$ with respect to the Manin-Schechtman partial order, i.e. $r_{\operatorname{max}}\geq r$ for all $r\in B(n,2;w)$.
\end{corollary}

By dualizing the arguments in Theorem~\ref{th:key}, we obtain the following.
\begin{theorem}
    Suppose $\rho \in A(n,2;w)$, $|\operatorname{Inv}(\rho)|\neq 0$. Then there exists $\rho'$ equivalent to $\rho$, $K \in \operatorname{Inv}(\rho)$, such that $P(K)$ forms a chain in $\rho'$. Moreover, $\operatorname{inv}_K(\gamma)=\rho'$ for some $\gamma \in A(n.2;w)$.
\end{theorem}
\begin{corollary}
    The class $r_{\operatorname{min}}$ is the unique minimal element in $B(n,2;w)$ with respect to the Manin-Schechtman partial order, i.e. $r_{\operatorname{min}}\leq r$ for all $r\in B(n,2;w)$.
\end{corollary}
\begin{corollary}
    Every $r\in B(n,2;w)$ sits on a maximal chain from $r_{\operatorname{min}}$ to $r_{\operatorname{max}}$, and the length of the chain is $|\I(w)|$.
\end{corollary}
\begin{example}
    Continuing from Example~\ref{ex:3}, we use the permutation $w=41758236$. A maximal chain is

    
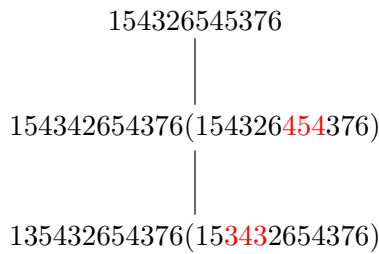
\begin{center}
\leavevmode
    \xymatrix{
    154326545376 \ar@{-}[d]\\
154342654376 (154326\textcolor{red}{454}376)\ar@{-}[d]\\
135432654376(15\textcolor{red}{343}2654376) }
\end{center}
   Here, we use the reverse Super-Yamanouchi word as the source, $\I(w)=\{\{7,5,2\},\{7,5,3\}\}$, hence the maximal chain has length $2$.
\end{example}

\bibliography{litlist} \label{references}
\bibliographystyle{amsalpha}

\end{document}